\documentclass[leqno]{amsart}

\usepackage{amsmath,amsfonts,amsthm,amssymb, mathrsfs, mathabx, accents, mathdots}

\usepackage[bookmarks]{hyperref}
\hypersetup{backref, colorlinks=true}

 \usepackage{graphicx}
\usepackage[all]{xy}
\newtheorem{theorem}{Theorem}[section]
\newtheorem{lemma}[theorem]{Lemma}
\newtheorem{proposition}[theorem]{Proposition}
\newtheorem{corollary}[theorem]{Corollary}
\newtheorem{remark}[theorem]{Remark}

\newtheorem{example}[theorem]{Example}
\newtheorem{definition}[theorem]{Definition}

\def\C{ \mathbb{C}}

\def\R{ \mathbb{R}}

\def\N{ \mathbb{N}}

\def\S{ \mathbb{S}}

\def\codim{ {\rm codim}}

\def\corank{ {\rm corank}}
\def\Reg{ {\rm Reg}}
\def\Rank{ {\rm Rank}}
\def\Grad{ {\rm Grad}}

\thanks{The research was partially supported by the post-doctoral FAPESP 
2013/18706-7 (for the first author) and  FAPESP Proc. 2014/00304-2 and CNPq Proc.
305651/2011-0 (for the second author).}

\def\rond{\mathaccent"7017}

\begin{document}

\large

\title[]{\small ON SINGULAR VARIETIES ASSOCIATED TO A POLYNOMIAL MAPPING \\ FROM $\C^n$ TO $\C^{n-1}$}
\makeatother

\author{ Nguyen Thi Bich Thuy and Maria Aparecida Soares Ruas}

\address[Maria Aparecida Soares Ruas]{Universidade de S\~ao Paulo, Instituto de Ci\^encias Matem\'aticas e de Computa{\c c}\~ao - USP, Avenida Trabalhador S\~ao-Carlense, 400 - Centro, Brazil.}
\email{maasruas@icmc.usp.br}

\address[Nguyen Thi Bich Thuy]{Ibilce-Unesp, Universidade Estadual Paulista ``J\'ulio de Mesquita Filho'',  Instituto de Bioci\^encias, Letras e Ci\^encias Exatas,  Rua Cristov\~ao Colombo, 2265, S\~ao Jos\'e do Rio Preto,  Brazil.}
\email{bichthuy@ibilce.unesp.br}

\maketitle \thispagestyle{empty}

\begin{abstract} 
We construct singular varieties ${\mathcal{V}}_G$ associated to a polynomial mapping $G : \C^{n} \to \C^{n - 1}$ where $n \geq 2$. 
Let  $G : \C^{3} \to \C^{2}$ be a local submersion, we prove that 
if the homology or the intersection homology with total perversity (with compact supports or closed supports) in dimension two  of any variety ${\mathcal{V}}_G$ is trivial then 
$G$ is  a fibration. In the case of a local submersion $G : \C^{n} \to \C^{n - 1}$ where $n \geq 4$, the result is still true with an additional  condition.

\end{abstract}

\section{INTRODUCTION}

Let $G: \C^n \to \C^{n-1}$ 
 be a non-constant polynomial mapping ($n \geq 2$). 
It is well known \cite{Thom} that $G$ is locally a trivial fibration outside the bifurcation set $B(G)$ in $\C^{n-1}$. In a natural way appears a fundamental question: how to determine the set $B(G)$. In \cite{VuiThang}, Ha Huy Vui and Nguyen Tat Thang gave, for a generic class of  $G: \C^n \to \C^{n-1}$ ($\, n \geq 2$), a necessary and sufficient condition for a point $z \in \C^{n-1}$ to be in the  bifurcation set $B(G)$ in term of the Euler characteristic of the fibers at  nearby points. The case n=2
was previously given in \cite{Vui-Trang} .

In this paper, we want to construct singular varieties ${\mathcal{V}}_G$ associated to a polynomial mapping $G : \C^{n} \to \C^{n - 1}$ ($n \geq 2$) such that the intersection homology of  ${\mathcal{V}}_G$ can 
cha\-rac\-te\-rize the bifurcation set of $G$. 
The motivation for this paper comes from the paper \cite{Valette}, where  Anna and Guillaume Valette constructed  real pseudomanifolds, denoted $V_F$, 
associated to a given polynomial mapping $F : \C^n \to \C^n$, such that 
the singular part of the variety $V_F$ is contained in $(S_F \times K_0(F)) \times \{ 0^p \}$ 
($\, p> 0$), where $K_0(F)$ is the set of critical values and $S_F$ is the set of non-proper points of $F$. 
 In the case of dimension $n = 2$, the homology or intersection homology of $V_F$ describes the geometry of the  singularities at infinity of 
 the mapping $F$. With Anna and Guillaume Valette, the first author generalized this result \cite{Thuy3} for the general case $F: \C^n \to \C^n$ ($\, n \geq 2$).
 The idea to construct varieties $V_F$ is the following: 
considering the polynomial mapping $F: \C^n \to \C^n$ as a real one $F: \R^{2n} \to \R^{2n}$, then if we take a finite covering $\{V_i\}$ by smooth submanifolds of $\R^{2n} \setminus SingF$, the mapping $F$ induces a diffeomorphism from $V_i$  into its image  $F(V_i)$. 
We use a technique  in order to separate these $\{F(V_i)\}$ by embedding them in a higher dimensional space, then $V_F$ is obtained by gluing $\{F(V_i)\}$ together along the set $S_F \cup K_0(F)$.

A natural question is that how can we apply this construction 
 to the case of polynomial mappings $G: \C^n \to \C^{n-1}$, or,  $G: \R^{2n} \to \R^{2n-2}$. 
The main difficulty of this case 
 is that if we take an open submanifold $V$ in $\R^{2n} \setminus SingF$, then locally we do not have a diffeomorphism from $V$ into its image $G(V)$. 
So we consider a generic $(2n -2)$- real dimensional submanifold in the source space $\R^{2n}$, denoted ${\mathcal{M}}_G$,  which is called {\it the Milnor set of $G$}. Then we can apply the construction of  singular varieties $V_F$ in \cite{Valette} for $F: = G\vert_{{\mathcal{M}}_G}$ the restriction of $G$ to the Milnor set   ${\mathcal{M}}_G$.

We obtain the following result: let  $G : \C^{3} \to \C^{2}$ be a local submersion, then 
if the homology or the intersection homology with total perversity (with compact supports or closed supports) in dimension two  of any  among of the constructed varieties ${\mathcal{V}}_G$ is trivial then 
$G$ is  a fibration (Theorem \ref{cidinhathuy1}). In the case of a local submersion $G : \C^{n} \to \C^{n - 1}$ where $n \geq 4$, the result is still true with an additional  condition (Theorem  \ref{cidinhathuy2}). Comparing with the paper \cite{VuiThang}, we obtain the Corollary \ref{coThuyCidinha3}.

\section{PRELIMINARIES AND BASIC DEFINITION}
In this section we set-up  our framework. All the varieties we consider in this article are semi-algebraic.

\subsection{Notations and conventions.}\label{section_notations}
Given a topological space $X$, singular simplices of $X$ will be semi-algebraic continuous mappings $\sigma:T_i \to X$, 
where $T_i$ is the standard $i$-simplex in $\R^{i+1}$. 
Given a subset $X$ of $\R^n$ we denote by $C_i(X)$ the group of
$i$-dimensional singular chains (linear combinations of singular simplices
with coefficients in $\R$); if
$c$ is an element of $C_i(X)$, we denote by $|c|$ its support.
By $Reg(X)$ and $Sing(X)$ we denote respectively the regular and singular locus of the set $X$. Given $X \subset \R^n$, $\overline{X}$ will stand for the topological closure of $X$. The smoothness to be considered as the differentiable smoothness.

Notice that, when we refer to the homology of a variety, the notation $H_*^{c}(X)$ refer to the homology with compact supports, the notation $H_*^ {cl}(X)$  refer to the homology with closed supports (see \cite{Jean Paul}).

\subsection{Intersection homology.}
We briefly recall the definition of intersection homology; for details,
we refer  to the fundamental work of M. Goresky and R. MacPherson
\cite{GM1} (see also \cite{Jean Paul}).

\begin{definition} 
{\rm Let $X$ be a $m$-dimensional semi-algebraic set.  A {\it semi-algebraic stratification of $X$} is the data of a finite semi-algebraic filtration 
$$X = X_{m} \supset X_{m-1} \supset \cdots \supset X_0 \supset X_{-1} = \emptyset,$$
 such that  for every $i$,  the set $S_i = X_i\setminus X_{i-1}$ is either empty or a manifold of dimension $i$. A connected component of $S_i$ is  called   {\it a stratum} of $X$.}
\end{definition}

We denote by $cL$  the open cone on the space $L$, the cone on the empty set being a point. Observe that if $L$ is a stratified set then $cL$ is stratified by  the cones over the strata of $L$ and a $0$-dimensional stratum (the vertex of the cone). 

\begin{definition}
{\rm A stratification of $X$ is said to be {\it locally topologically trivial} if for every $x \in X_i\setminus X_{i-1}$, $i \ge 0$, there is an open neighborhood $U_x$  of $x$ in $X$, a stratified set $L$ and a semi-algebraic homeomorphism 
 $$h:U_x \to (0;1)^i \times  cL,$$   such  that
 $h$ maps the strata of $U_x$ (induced stratification) onto the strata of  $  (0;1)^i \times cL$  (product stratification).}
\end{definition}

The definition of perversities as originally given by Goresky and MacPherson:
\begin{definition}
{\rm 
A {\it perversity} is an $(m+1)$-uple  of integers $\bar p = (p_0, p_1, p_2,
p_3,\dots , p_m)$ such that $p_0 = p_1 = p_2 = 0$ and $p_{k+1}\in\{p_k, p_k + 1\}$, for $k \geq 3$.

Traditionally we denote the zero perversity by
$\overline{0}=(0, 0, 0,\dots,0)$, the maximal perversity by
$\overline{t}=(0, 0, 0,1,\dots,m-2)$, and the middle perversities by
$\overline{m}=(0, 0, 0,0,1,1,\dots, [\frac{m-2}{2}])$ (lower middle) and
$\overline{n}=(0, 0, 0,1,1,2,2,\dots ,[\frac{m-1}{2}])$ (upper middle). We say that the
perversities $\overline{p}$ and $\overline{q}$ are {\it complementary} if $\overline{p}+\overline{q}=\overline{t}$.

Let $X$ be a semi-algebraic variety such that $X$ admits a  locally topologically trivial stratification. We say that a semi-algebraic subset $Y\subset X$ is  $(\bar
p, i)$-{\it allowable} if  
$$\dim (Y \cap X_{m-k}) \leq i - k + p_k \text{ for
all } k.$$

Define $IC_i ^{\overline{p}}(X)$ to be the $\R$-vector subspace of $C_i(X)$
consisting in those chains $\xi$ such that $|\xi|$ is
$(\overline{p}, i)$-allowable and $|\partial \xi|$ is
$(\overline{p}, i - 1)$-allowable.}
\end{definition}

\begin{definition} 
{\rm The {\it $i^{th}$ intersection homology group with perversity $\overline{p}$}, denoted by
$IH_i ^{\overline{p}}(X)$, is the $i^{th}$ homology group of the
chain complex $IC^{\overline{p}}_*(X).$
}
\end{definition}

Notice that, the notation $IH_*^{\overline{p}, c}(X)$ refer to the intersection homology with compact supports, the notation $IH_*^ {\overline{p},cl}(X)$  refer to the intersection homology with closed supports.

 Goresky and MacPherson proved that the intersection homology  is independent of the
choice of the stratification \cite{GM1, GM2}.

The Poincar\'e duality holds for the intersection homology of a (singular) variety:

 \begin{theorem}[Goresky,
MacPherson \cite{GM1}] {\it For any orientable compact stratified
semi-algebraic $m$-dimensional variety  $X$,  generalized Poincar\'e duality holds:
$$IH_k ^{\overline{p}}(X)  \simeq IH_{m-k} ^{\overline{q}} (X),$$
where $\overline{p}$ and $\overline{q}$ are complementary perversities.
}
\end{theorem}

\noindent For the non-compact case, we have:
$$IH_k ^{\overline{p}, c}(X)  \simeq IH_{m-k} ^{\overline{q}, cl} (X).$$
A relative version is also true in the case where $X$ has boundary.

\subsection{The bifurcation set, the set of asymptotic critical values and the asymptotic set} 
Let $G : \C^n \to \C^m$ where $n \geq m$ be a polynomial mapping. 

\medskip

i) The bifurcation set of $G$, denoted by $B(G)$ is the smallest set in $\C^m$ such that $G$ is not $C^{\infty}$ - fibration on this set (see, for example, \cite{Thom}). 

\medskip 

ii) The set of asymptotic critical values, denoted by $K_{\infty}(G)$, is the set 
$$K_{\infty}(G) = \{ \alpha \in \C^m : \exists \{ z_k\} \subset \C^n, \text{ such that } \vert z_k \vert \to \infty, 
G(z_k) \to \alpha  \text{ and } \vert z_k \vert  \vert dG(z_k) \vert \to 0 \}.$$
The set $K_{\infty}(G)$ is an approximation of the set $B(G)$. More precisely, we have $B(G) \subset K_{\infty}(G)$ (see, for example, \cite{Ku} or \cite{Luis}).

\medskip 

iii) When $n = m$, we denote by $S_G$ the set of points at which the mapping $G$ is not proper, {\it i.e.} 
$$S_G= \{ \alpha \in \C^m : \exists \{ z_k\} \subset \C^n, \vert z_k \vert \to \infty \text{ such that } G(z_k) \to \alpha\},$$
and call it the {\it asymptotic variety}. In the case of polynomial mappings $F: \C^n \to \C^n$, the following holds: $B(G) = S_G$ (\cite{Jelonek2}).

\section{THE VARIETY ${\mathcal{M}}_G$} 

We consider polynomial mappings $G : \C^n \to \C^{n-1}$ as real ones $G : \R^{2n} \to \R^{2n-2}$. By $Sing(G)$ we mean the singular locus of $G$, that   is the set of points for which the (complex) rank of the Jacobian matrix is  less than $n - 1$. We denote by $K_0(G)$ the set of critical values. From here, we assume always $K_0(G) = \emptyset$.

\begin{definition} \label{defVG}
{\rm Let $G : \C^n \to \C^{n-1}$ be a polynomial mapping. Consider $G$ as a real polynomial mapping  $G : \R^{2n} \to \R^{2n-2}$. Let $\rho : \C^n \to \R$ be a real function such that $\rho(z) \geq 0$ for any $z \in \C^n$.
Let 
$$\varphi = \frac{1}{1 + \rho}.$$
Consider $(G, \varphi)$  as a mapping from $\R^{2n}$ to $\R^{2n-1}$. 
 Let us define
$${\mathcal{M}}_G : = Sing(G, \varphi) = \{x \in \R^{2n} \text{ such that } \Rank D_{\R}(G, \varphi)(x) \leq 2n - 2\},$$
where $D_{\R}(G, \varphi)(x)$ is the Jacobian matrix of $(G, \varphi): \R^{2n} \to \R^{2n - 1}$ at $x$.
}
\end{definition}

\begin{remark}
{\rm Since $K_0(G) = \emptyset$, then $\Rank D_{\R}(G) = 2n -2$, so we have 
$$Sing(G, \varphi) = \{x \in \R^{2n} \text{ such that } \Rank D_{\R}(G, \varphi) = 2n - 2\}.$$
}
\end{remark}

Note that, from here, if we want to refer to the source space as a complex space, we will write  $(G, \varphi) : \C^{n} \to \R^{2n-1}$, if we want to refer to the source space as a real space, we will write $(G, \varphi) : \R^{2n} \to \R^{2n-1}.$ Moreover, in general, we denote by $z$  a complex element in $\C^n$ and by $x$ a real element in $\R^{2n}$.

\begin{lemma} \label{lemmaSingGpsi}
{\rm For any $\rho, \varphi$ and $(G, \varphi)$ as in the Definition \ref{defVG} and for any $x = (x_1, \ldots, x_{2n}) \in \R^{2n}$, we have   
$$\Rank D_{\R}(G, \varphi)(x) = \Rank D_{\R}(G, \rho)(x),$$ 
so we have 
$${\mathcal{M}}_G = Sing(G, \varphi) = Sing(G, \rho).$$}
\end{lemma}

\begin{proof}
For any $x = (x_1, \ldots, x_{2n}) \in \R^{2n}$, we have 
$$ D_{\R}(G, \rho) (x)={ \begin{pmatrix}  & D_{\R}(G) & \\ \rho_{x_1} & \ldots & \rho_{x_{2n}}  \end{pmatrix}},$$
$$ D_{\R}(G, \varphi) (x)={ \begin{pmatrix}  & D_{\R}(G) & \\ \frac{-\rho_{x_1}}  { {(1 + \rho)}^2 } & \ldots & \frac { -\rho_{x_{2n} } } { {(1 + \rho)}^2}   \end{pmatrix} },$$
where $\rho_{x_i}$ is the derivative of $\rho$ with respect to $x_i$, for $i = 1, \ldots, 2n.$
We have $\Rank D_{\R}(G, \varphi) (x)= \Rank D_{\R}(G, \rho)(x)$ for any $x \in \R^{2n}$ and ${\mathcal{M}}_G = Sing(G, \varphi) = Sing(G, \rho)$.

\end{proof}

\begin{remark} \label{remark rho}
{\rm From here, we consider the function $\rho$ of the following form
$$\rho = a_1 \vert z_1 \vert ^2 + \cdots + a_n \vert z_n \vert^2,$$
 where $\Sigma_{i = 1, \ldots, n} a^2_i \neq 0 $, $\, a_i \geq 0$, and $a_i \in \R$ for $i = 1, \ldots, n.$
}
\end{remark}

\begin{proposition} \label{remarkmixedfunction}
{\rm Let 
  $G = (G_1, \ldots, G_{n-1}) : \C^n \to \C^{n -1} \, ( n \geq 2)$ be a polynomial mapping such that $K_0(G) = \emptyset$ and $\rho: \C^n \to \R$ be such that 
$\rho = a_1 \vert z_1 \vert ^2 + \cdots + a_n \vert z_n \vert^2,$
 where $\Sigma_{i = 1, \ldots, n} a^2_i \neq 0$, $a_i \geq 0$ and $a_i \in \R$, for $i= 1, \ldots, n$. Denote by  ${\bf v}_i$ the determinant of the cofactor of  $\frac{\partial}{\partial z_i}$ of the matrix  
$${\bf V}(z) =
\begin{pmatrix}
\frac{\partial}{\partial z_1} & \cdots &  \frac{\partial}{\partial z_n} \\
\frac{\partial G_1}{\partial z_1} & \cdots &  \frac{\partial G_1}{\partial z_n} \\
& \cdots \\
\frac{\partial G_{n-1}}{\partial z_1} & \cdots &  \frac{\partial G_{n-1}}{\partial z_n} \\
  \end{pmatrix},$$
 for $i = 1, \ldots, n$. Then  we have 
$${\mathcal{M}}_G = h^{-1}(0),$$
 where $$h : \C^{n} \to \C, \quad h(z) = 2\Sigma a_i{\bf v}_i(z) \overline{z_i}.$$
}
\end{proposition}

\begin{proof}

Let $G = (G_1, \ldots, G_{n-1}) : \C^n \to \C^{n -1} \, ( n \geq 2)$ and $\rho: \C^n \to \R$ such that 
$\rho = a_1 \vert z_1 \vert ^2 + \cdots + a_n \vert z_n \vert^2,$
 where $\Sigma_{i = 1, \ldots, n} a^2_i \neq 0$, $a_i \geq 0$ and $a_i \in \R$. Let us consider the vector field 

$${\bf V}(z) =
\begin{pmatrix}
\frac{\partial}{\partial z_1} & \cdots &  \frac{\partial}{\partial z_n} \\
\frac{\partial G_1}{\partial z_1} & \cdots &  \frac{\partial G_1}{\partial z_n} \\
& \cdots \\
\frac{\partial G_{n-1}}{\partial z_1} & \cdots &  \frac{\partial G_{n-1}}{\partial z_n} \\
  \end{pmatrix}.$$
We have 
$${\bf V}(z) = {\bf v}_1 \frac{\partial}{\partial z_1} + \cdots + {\bf v}_n \frac{\partial}{\partial z_n},$$ 
 where ${\bf v}_i$ is the determinant of the cofactor of  $\frac{\partial}{\partial z_i}$, for $i = 1, \ldots, n$.
The vector field ${\bf V}(z)$ 
is tangent to the curve $G = c$. Let $R(z) = a_1  z_1^2  + \cdots + a_n z_n^2$, then 
 we have ${\mathcal{M}}_G = h^{-1}(0)$, where 
$$h : \C^n \to \C, \quad h(z) = <{\bf V}(z), \Grad \, R(z)>.$$
More precisely,  we have $h(z) = 2\Sigma a_i{\bf v}_i(z) \overline{z_i}.$
\end{proof}

\begin{proposition} \label{prodimSingGpsi}
{\rm For an open and dense set of polynomial mappings $G: \C^n \to \C^{n-1}$ such that $K_0(G) = \emptyset$, the variety ${\mathcal{M}}_G$  is a smooth manifold of  dimension $2n-2$. 
}
\end{proposition}

\begin{proof} 
The question is of local nature. In a neighbourhood of a point $z_0$ in $\C^n$, we can choose coordinates such that the level curve $G = c$, where $c = G(z_0) \in \C^{n - 1}$ is parametrized 
$$\gamma : (\C, 0) \to (\C^n, z_0)$$
$$ \quad \quad \quad \quad \quad \quad \quad \quad s \mapsto (\gamma_1(s), \ldots, \gamma_n(s)).$$ 
Since $\rho = a_1 \vert z_1 \vert ^2 + \cdots + a_n \vert z_n \vert ^2$, then $\rho \circ \gamma : (\C, 0) \to \R$ and 
$$\rho \circ \gamma (s) = a_1 \vert \gamma_1(s) \vert ^2 + \cdots + a_n \vert \gamma_n(s) \vert ^2.$$
If $z_0$ is a singular point of $\rho \vert_{G= c}$, then 
$$\rho \circ \gamma(0) = \rho(\gamma(0)) = \rho(z_0),$$
$$(\rho \circ \gamma)'(0) = 0.$$
For an open and dense set of $G$, we have 
$$(\rho \circ \gamma)''(0) \neq 0.$$
Hence, $z_0$ is a Morse singularity of $\rho \vert_{G= c}$. 
In particular, it is an isolated point of the level curve $G = c$. When $c$ varies in $\C^{n-1}$, it follows that the set ${\mathcal{M}}_G$  has dimension $2n-2$. 

We prove now that ${\mathcal{M}}_G$ is smooth. By Proposition \ref{remarkmixedfunction}, the variety ${\mathcal{M}}_G$ is the set of solutions of $h = 0$, where 
$$h(z) = 2 \Sigma a_i {\bf v}_i(z) \overline{z_i},$$
and   ${\bf v}_i$ is the determinant of the cofactor of  $\frac{\partial}{\partial z_i}$ of 
${\bf V}(z)$, 
for $i = 1, \ldots, n$. 
Since $K_0(G) = \emptyset$ then ${\bf V}(z) = ({\bf v}_1(z), \ldots, {\bf v}_n(z))\neq 0$. 
We can assume that ${\bf V}(z_0) \neq 0$ for a fixed point $z_0$.  
For a generic polynomial mapping, we can solve the equation $h = 0$ in a neighbourhood of $z_0$. This shows that $h = 0$ is smooth in a neighbourhood of $z_0$. Then $M_G$ is smooth.

\end{proof}

\begin{remark}
{\rm From here, we consider always generic polynomial mappings $G: \C^n \to \C^{n - 1}$ as in the Propostion  \ref{prodimSingGpsi}.
}
\end{remark}

\section{The variety ${\mathcal{V}}_G$}
\subsection{The construction of the variety ${\mathcal{V}}_G$} \label{VG}
Let $G: \C^n \to \C^{n-1}$ and $\rho, \varphi: \C^n \to \R$ such that 
$$\rho = a_1 \vert z_1 \vert ^2 + \cdots + a_n \vert z_n \vert^2, \quad \varphi = \frac{1}{1 + \rho},$$
where $\Sigma_{i = 1, \ldots, n} a^2_i \neq 0$, $\, a_i \geq 0$ and $a_i \in \R$.
 Let us consider:

\medskip

a) $F : = G_{\vert {\mathcal{M}}_G }$ the restriction of $G$ on ${\mathcal{M}}_G$,

\medskip 

b) ${\mathcal{N}}_G = {\mathcal{M}}_G \setminus F^{-1}(K_0(F))$.

\medskip

\noindent  Since the dimension of ${\mathcal{M}}_G$  is $2n - 2$ (Proposition \ref{prodimSingGpsi}), then locally, in a neighbourhood of any point $x_0$ in ${\mathcal{M}}_G$, 
   we get a mapping $F : \R^{2n-2}  \to \R^{2n - 2}$. Now, we can apply the construction of singular varieties $V_F$ in \cite{Valette} for $F : = G_{\vert {\mathcal{M}}_G }$: 
there exists a co\-ver\-ing $\{ U_1, \ldots , U_p \}$ of ${\mathcal{N}}_G$  by open semi-algebraic subsets (in $\R^{2n}$) such that on every element of this co\-ver\-ing, the mapping $F$ induces a diffeomorphism   
onto its image (see Lemma 2.1 of \cite{Valette}, see also \cite{Thuy1}). We can find semi-algebraic closed  subsets $V_i \subset U_i$ (in ${\mathcal{N}}_G$) which cover ${\mathcal{N}}_G$ as well. 
Thanks to Mostowski's Separation Lemma (see Separation Lemma in \cite{Mos}, page 246), for each $\, i =1, \ldots , p$, 
there exists a Nash function $\psi_i : {\mathcal{N}}_G \to \R$,  
such that  $\psi_i$ is positive on $V_i$ and negative on ${\mathcal{N}}_G \setminus U_i$. 

\begin{lemma}
{\rm We can choose the Nash functions $\psi_i$ such that $\psi_i (x_k)$ tends to  zero when $\{x_k\} \subset {\mathcal{N}}_G$ tends to infinity.
}
\end{lemma}   

\begin{proof}
{\rm 
If $\psi_i$ is a Nash function, then with any $N_i \in (\N \setminus \{ 0 \})$, the function  
$$\psi'_i (x) = \frac{\psi_i (x)}{{(1 + \vert x \vert^2)}^{N_i}},$$
where $x \in {\mathcal{N}}_G$, is also a Nash function, for $i = 1, \ldots, p.$ The Nash function $\psi'_i$ satisfies the property: $\psi'_i$ is positive on $V_i$ and negative on ${\mathcal{N}}_G \setminus U_i$. With $N_i$ large enough,  $\psi'_i (x_k)$ tends to  zero when $\{x_k\} \subset {\mathcal{N}}_G$ tends to infinity, for $i = 1, \ldots, p$. We replace the function $\psi_i$ by $\psi'_i$.
}
\end{proof}

\begin{definition} \label{defVGnew}
{\rm 
Let the Nash functions $\psi_i$ and $\rho$ be such that $\psi_i (x_k)$ tends to  zero and $\rho(x_k)$ tends to infinity  when $x_k \subset {\mathcal{N}}_G$ tends to infinity. 
Define a variety ${\mathcal{V}}_G$ associated to $(G, \rho)$ as 
$${\mathcal{V}}_G : = \overline{(F, \psi_1, \ldots, \psi_p)({\mathcal{N}}_G)}.$$ 
}
\end{definition}

\begin{remark}
{\rm 
For a given polynomial mapping $G: \C^n \to \C^{n-1}$, the variety ${\mathcal{V}}_G$ is not unique. It depends on the choice of the function $\rho$ and the Nash functions $\psi_i$.
}
\end{remark}
\begin{proposition} \label{ProdimVG}
{\rm The real dimension of ${\mathcal{V}}_G$ is $2n - 2$.   
}
\end{proposition}
\begin{proof}
By Proposition \ref{prodimSingGpsi}, in the generic case, the real dimension of ${\mathcal{M}}_G$ is $2n - 2$. Moreover, $F$ is  a  local immersion in a neighbourhood of a point in  ${\mathcal{M}}_G$. So, the real dimension of $F({\mathcal{M}}_G)$ is also $2n - 2$. Since  
$$F({\mathcal{N}}_G) = F({\mathcal{M}}_G) \setminus K_0(F),$$
so the real dimension of $F({\mathcal{N}}_G)$ is $2n - 2$. By Definition \ref{defVGnew}, the real dimension of ${\mathcal{V}}_G$ is $2n - 2$.

\end{proof}

\begin{definition} [see, for example, \cite{Cidinha}] \label{defSG}
{\rm 
Let $G : \C^n \to \C^{n-1}$ be a  polynomial mapping and $\rho : \C^n \to \R$  a real function such that $\rho \geq 0$. Define 
$${\mathcal{S}}_G:=\{ \alpha \in \C^{n-1}: \exists \{z_k\} \subset Sing(G, \rho),  \text{ such that } z_k \text{ tends to infinity}, G(z_k) \text{ tends to } \alpha \}.$$
}
\end{definition}

\begin{remark}  \label{remarkSG}
{\rm 
 a) For any real function $\rho : \C^n \to \R$ such that $\rho \geq 0$, we have 
$$B(G) \subset {\mathcal{S}}_G \subset K_{\infty}(G),$$
where $B(G)$ is the bifurcation set  and $K_{\infty}(G)$ is the set of asymptotic critical values of $G$ (see, for example, \cite{Cidinha}).

\medskip 

b) By Lemma \ref{lemmaSingGpsi}, we have $Sing(G, \rho) = {\mathcal{M}}_G,$ 
so the set ${\mathcal{S}}_G$ can be written
$${\mathcal{S}}_G:=\{ \alpha \in \C^{n-1}: \exists \{x_k\} \subset {\mathcal{M}}_G,  \text{ such that }x_k \text{ tends to infinity}, G(x_k) \text{ tends to } \alpha \}.$$
}
\end{remark}

\begin{definition} \label{defsiginfi}
{\rm The {\it singular set at infinity} of the variety ${\mathcal{V}}_G$ is the set 
$$\{ \beta \in {\mathcal{V}}_G: \exists \{x_k \} \subset {\mathcal{N}}_G, x_k \to \infty, (G, \psi_1, \ldots, \psi_p)(x_k) \to \beta \}.$$

}
\end{definition}

\begin{proposition} \label{singularVG}
{\rm The singular set at infinity 
of the variety ${\mathcal{V}}_G$  is contained in the set ${\mathcal{S}}_G \times \{0_{\R^{ p}}\}.$ 
}
\end{proposition}

\begin{proof}

At first, by Proposition \ref{prodimSingGpsi}, for the generic case, the real dimension of ${\mathcal{V}}_G$ associated to $G: \C^n \to \C^{n-1}$ is $2n - 2$. Moreover, we have the following facts:

a) $ {\mathcal{S}}_G \subset K_{\infty}(G),$

b) $\dim_{\C} (K_{\infty}(G)) \leq n-2$ (see \cite{Ku}), so $\dim_{\R} (K_{\infty}(G)) \leq 2n-4$.

Hence, we have $\dim_{\R}{\mathcal{S}}_G \times \{0_{\R^{ p}}\}  \leq 2n - 4$. Moreover, by Proposition \ref{ProdimVG}, we have  $\dim_{\R}{\mathcal{V}}_G = 2n -2$. Let $\beta$ be a singular point at infinity of the variety  ${\mathcal{V}}_G$, 
then 
there exists a sequence  $\{x_k\}$ in ${\mathcal{N}}_G$ tending to infinity 
such that $(G, \psi_1, \ldots, \psi_p)(x_k)$ tends to $\beta$. 
Assume that $G(x_k)$ tends to $\alpha$, then $\alpha$ belongs to  ${\mathcal{S}}_G$. Moreover, the Nash function $\psi_i(x_k)$ tends to 0, for any $i = 1, \ldots, p$. So $\beta=(\alpha, 0_{\R^{p}})$ belongs to ${\mathcal{S}}_G \times \{0_{\R^{p}}\}.$ 
Notice that, by Definition of ${\mathcal{V}}_G$, the set ${\mathcal{S}}_G \times \{0_{\R^{p}}\}$ is contained in ${\mathcal{V}}_G$. Then ${\mathcal{S}}_G \times \{0_{\R^{p}}\}$ contains the singular set at infinity of the variety ${\mathcal{V}}_G$.

\end{proof}

\begin{remark} 
{\rm The singular set at infinity of ${\mathcal{V}}_G$ depends on the choice of the function $\rho$, since when  $\rho$ changes, the set ${\mathcal{S}}_G$ also changes. But, the property 
$B(G) \subset {\mathcal{S}}_G$ does not depend on the choice of the function $\rho$ (see, for example, \cite{Cidinha}).
}
\end{remark}

The previous results show the following Proposition:

\begin{proposition}
{\rm Let $G: \C^n \to \C^{n-1}$ be a polynomial mapping such that $K_0(G) = \emptyset$ and let $\rho: \C^n \to \R$ be a real function such that 
$$\rho = a_1 \vert z_1 \vert ^2 + \cdots + a_n \vert z_n \vert^2,$$
where $\Sigma_{i = 1, \ldots, n} a^2_i \neq 0$, $\, a_i  \geq 0$ and $a_i \in \R$ for $i = 1, \ldots, n.$ Then, there exists a real variety ${\mathcal{V}}_G$ in $\R^{2n - 2 + p}$, where $p > 0$, such that: 
 
1) The real dimension of ${\mathcal{V}}_G$ is $2n -2$,

2) The singular set at infinity 
of the variety ${\mathcal{V}}_G$  is contained in ${\mathcal{S}}_G \times \{0_{\R^{p}}\}.$
}
\end{proposition}

\begin{remark}
{\rm 
The variety ${\mathcal{V}}_G$ depends on the choice of the function $\rho$ and the functions $\psi_i$. 
Form now, we denote by ${\mathcal{V}}_G(\rho)$ any variety ${\mathcal{V}}_G$ associated to $(G, \rho)$. If we refer to ${\mathcal{V}}_G$, that means a variety ${\mathcal{V}}_G$ associated to $(G, \rho)$ for any $\rho$.
}
\end{remark}

\begin{remark}
{\rm 
1) In the construction of singular varieties ${\mathcal{V}}_G$, we can put $F : = (G, \varphi)_{\vert {\mathcal{M}}_G }$, that means $F$ is the restriction of $(G,\varphi)$ on ${\mathcal{M}}_G$. In this case, since the dimension of ${\mathcal{M}}_G$  is $2n - 2$  then locally, in a neighbourhood of any point $x_0$ in ${\mathcal{M}}_G$, 
   we get a mapping $F : \R^{2n-2}  \to \R^{2n - 1}$. The construction of singular varieties ${\mathcal{V}}_G$ can be applied also in this case. 

2) The construction of singular varieties ${\mathcal{V}}_G$ can be applied for polynomial mappings $G: \C^n \to \C^p$ where $p < n -1$ if the Milnor set ${\mathcal{M}}_G$ is smooth is this case. 

}
\end{remark}

\subsection{A variety ${\mathcal{V}}_G$ in the case of the Broughton's Example}
\begin{example} \label{exBroughton}
{\rm We compute in this example a variety ${\mathcal{V}}_G$ in the case of the Broughton's example \cite{Broughton}:
$$G: \C^2 \to \C, \quad \quad \quad  G(z, w) = z + z^2w.$$
 We see that $K_0(G) = \emptyset$ since the system of equations
$G_z = G_w = 0$
has no solutions. 

\noindent Let us denote 
$$z = x_1 + ix_2, \quad \quad \quad w=x_3 + ix_4,$$
where $x_1, x_2, x_3, x_4 \in \R$. Consider $G$ as a real polynomial mapping, we have 
$$G(x_1, x_2, x_3, x_4) = (x_1 + x_1^2 x_3 - x_2^2x_3 - 2x_1x_2x_4, x_2 + 2x_1x_2x_3+x_1^2x_4 - x_2^2x_4).$$
Let $\rho = \vert w \vert^2$, then 
$$\varphi = \frac{1}{1 + \rho} = \frac{1}{1 + \vert w \vert^2} = \frac{1}{1 + x_3^2 + x_4^2}.$$
The Jacobian matrix of $(G, \rho)$ is 
$$D_{\R}(G, \rho) ={ \begin{pmatrix} 1 + 2x_1x_3 - 2 x_2x_4 & -2x_2x_3 - 2x_1x_4 & x_1^2 - x_2^2 & -2x_1x_2  \\  2x_2x_3 + 2x_1x_4 & 1 + 2x_1x_3 - 2x_2x_4 & 2x_1x_2 & x_1^2 - x_2^2 \\ 0 & 0 & 2x_3 & 2x_4\end{pmatrix}}.$$
By an easy computation,  we have ${\mathcal{M}}_G = Sing(G, \rho)  = M_1 \cup M_2,$
where 

$M_1 := \{ (x_1, x_2, 0, 0) : x_1, x_2 \in \R \}$,

$ M_2 = \left\{ (x_1, x_2, x_3, x_4) \in \R^4 : 1+2x_1x_3 - 2x_2x_4 = 2x_2x_3 + 2x_1x_4 = 0  \right\}.$

\noindent Let us consider $G$ as a real mapping from $\R^4_{(x_1, x_2, x_3, x_4)}$ to  $\R^2_{(\alpha_1, \alpha_2)}$, then:

\medskip 

a) If $x = (x_1, x_2, 0, 0) \in M_1$, we have $G(x) = (x_1, x_2)$.

\medskip 

b) If $x = (x_1, x_2, x_3, x_4) \in M_2$, then we have 
 $ G(x) = (\alpha_1, \alpha_2), $ 
 where 
$$\alpha_1 = \frac{-x_3}{4(x_3^2 + x_4^2)}, \quad \quad \quad \alpha_2 = \frac{x_4}{4(x_3^2 + x_4^2)}.$$ 
Let $F:= G_{\vert {\mathcal{M}}_G}$. We can check easily that $K_0(F) = \emptyset$. 
Choosing  ${\mathcal{M}}_G$ as a  covering of  ${\mathcal{M}}_G$ itself. We choose the Nash function $\psi = \varphi$, then $\psi$ is positive on all ${\mathcal{M}}_G$. So, by Definition \ref{defVGnew}, we have 
$${\mathcal{V}}_G= \overline{(F, \varphi)({\mathcal{M}}_G}) = \overline{(G, \varphi)({\mathcal{M}}_G}) =  (G, \varphi)(M_1) \cup (G, \varphi)(M_2) \cup ({\mathcal{S}}_G \times {0_{\R}}),$$ 
where $(G, \varphi,): \R^4_{(x_1, x_2, x_3, x_4)} \to \R^3_{(\alpha_1, \alpha_2, \alpha_3)}$. Then 

\medskip 

a) $(G, \varphi)(M_1)$ is the plane $\{\alpha_3 = 1\} \subset \R^3_{(\alpha_1, \alpha_2, \alpha_3,)}.$

\medskip 

b) Assume that $(\alpha_1, \alpha_2, \alpha_3) \in (G, \varphi)(M_2)$, and let 
$$x_3 = r cos \theta, \quad \quad \quad  x_4 = r sin \theta,$$
where $r \in \R, \, r>0$ and $0 \leq \theta \leq 2\pi,$ then 
$$\alpha_1^2 + \alpha_2^2 = \frac{1}{16r^2}, \quad \quad \quad \alpha_3 =  \frac{1}{1 + r^2}.$$
So $(G, \varphi)(M_2)$ is a 2-dimensional open cone. In fact, when $r$ tends to infinity, then $\alpha_1, \alpha_2$  and $\alpha_3$ tend to 0, but the origin does not belong to this cone.

\medskip

Moreover, by an easy computation, we can verify that  the set ${\mathcal{S}}_G$ is $0 = (0, 0) \in \R^2_{(\alpha_1, \alpha_2)}$. So the origin 0 of $ \R^3_{(\alpha_1, \alpha_2, \alpha_3)}$ belongs to ${\mathcal{V}}_G$. In conclusion, the variety ${\mathcal{V}}_G$ is the union of the plane $\alpha_3 = 1$ and 
 a 2-dimensional cone $\mathcal{C}$  with vertex $0$, where the cone $\mathcal{C}$ tends to infinity 
 and asymptotic to the plane $\alpha_3 = 1$ in $ \R^3_{(\alpha_1, \alpha_2, \alpha_3)}$ (see Figure \ref{FigureExBr}). 

\begin{figure}[h!] 
\begin{center}
\includegraphics[scale=0.5]{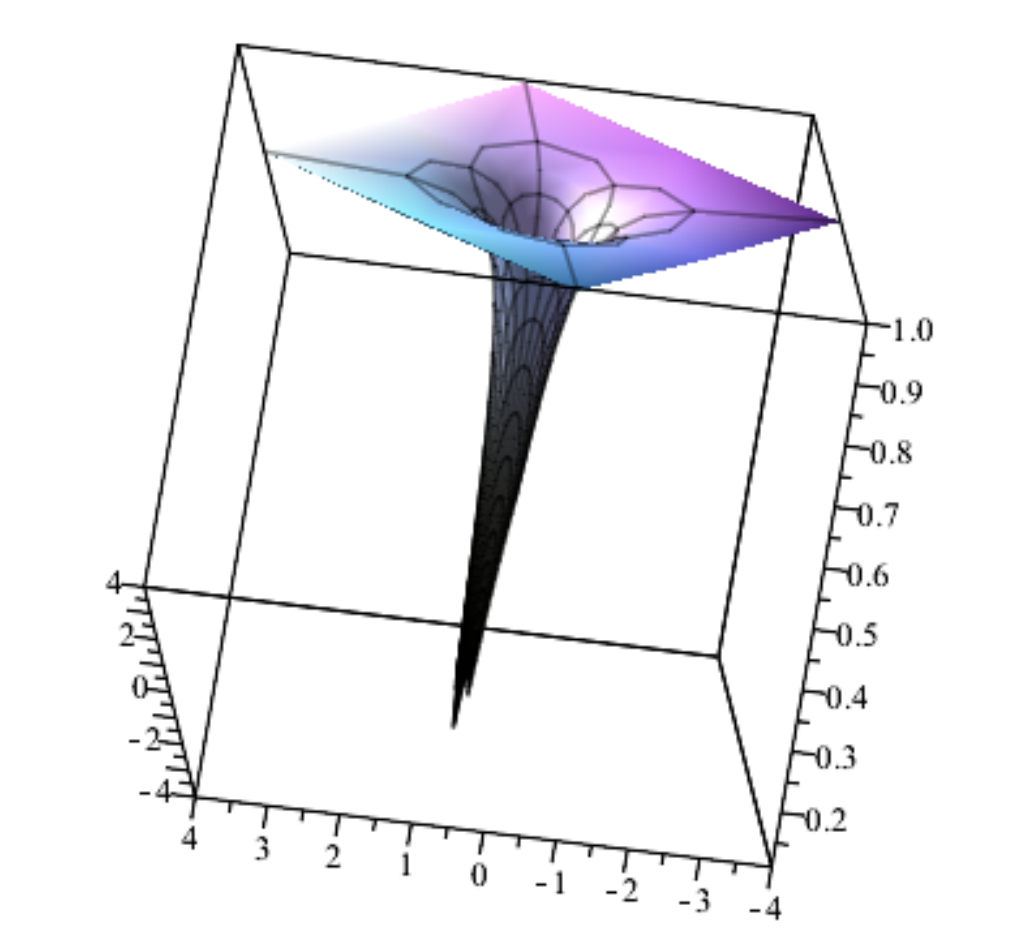}
\caption{A variety  ${\mathcal{V}}_G$ in the case of  the Broughton's Example $G(z, w) = z + z^2w$.}
\label{FigureExBr}
\end{center}\end{figure}

}
\end{example}

\begin{remark}
{\rm We can use the Proposition \ref{remarkmixedfunction} with the view of mixed functions (see \cite{oka}) to determine the variety ${\mathcal{M}}_G$. Let us return to the Example \ref{exBroughton}. In this case $\rho = \vert w \vert^2$, then 
$${\mathcal{M}}_G = \left \{ (z, w) \in \C^2 :  \frac{\partial G}{\partial z} \overline{ w} = 0 \right \}.$$
Hence we have $(1 + 2 zw) \overline{w} = 0$, that implies the following two cases:

\medskip 

i) $\overline{w} = 0$:  We have $x_3 = x_4 = 0$, where $w = x_3 + ix_4.$

\medskip 

ii) $\overline{w} \neq 0$ and $z = - \frac{1}{2w} = - \frac{\overline{w}}{2\vert w \vert^2} $: We have 
$$x_1 = \frac{-x_3}{2(x_3^2 + x_4^2)}, \quad \quad \quad x_2 = \frac{x_4}{2(x_3^2 + x_4^2)} ,$$
where $z = x_1 + ix_2$. 

So we get ${\mathcal{M}}_G = M_1 \cup M_2$ as the computations and notations in the Example  \ref{exBroughton}.

\medskip

}
\end{remark}

\section{Results}

\begin{theorem} \label{cidinhathuy1}
{\rm 
\medskip 

Let $G = (G_1, G_2): \C^3 \rightarrow \C^2$ be a polynomial mapping such that $K_0(G) = \emptyset$. If one the groups $IH_2^{\overline{t}, c}({\mathcal{V}}_G, \R)$, $\,IH_2^{\overline{t}, cl}({\mathcal{V}}_G, \R)$, $\, H_2^c({\mathcal{V}}_G, \R)$   and $H_2^{cl}({\mathcal{V}}_G, \R)$  is trivial then the bifurcation set $B(G)$ is empty.  
}
\end{theorem}

\begin{theorem} \label{cidinhathuy2}
{\rm 
\medskip 

Let $G = (G_1, \ldots, G_{n-1}): \C^n \rightarrow \C^{n-1}$ $(n \geq 4)$ be a polynomial mapping such that $K_0(G) = \emptyset$ and $\Rank_{\C} {(D \hat {G_i})}_{i=1,\ldots,n-1} > n-3$, where $\hat {G_i}$ is the leading form of $G_i$, that is the homogenous part of highest degree of $G_i$,  for $i = 1, \ldots, n-1$.  Then if one the groups 
$IH_2^{\overline{t}, c}({\mathcal{V}}_G, \R)$, $\, IH_2^{\overline{t}, cl}({\mathcal{V}}_G, \R)$, 
 $\, H_2^{c}({\mathcal{V}}_G, \R)$, $\, H_2^{cl}({\mathcal{V}}_G, \R)$,   
 $\, H_{2n-4}^{c}({\mathcal{V}}_G, \R)$   and  $H_{2n-4}^{cl}({\mathcal{V}}_G, \R)$ is trivial then the bifurcation set $B(G)$ is empty.
}
\end{theorem}


Before proving these Theorems, we recall some necessary Definitions and Lemmas.

\begin{definition}
A {\it semi-algebraic family of sets (parametrized by $\R$)} 
{\rm is a semi-algebraic set $A \subset \R^n \times \R$, the last variable being considered as parameter.}
\end{definition}
\begin{remark}
{ \rm A semi-algebraic set $A \subset \R^n \times \R$ will be considered as a family parametrized by $t \in \R$. We write $A_t$, for ``the fiber of $A$ at $t$'', {\it i.e.}:
$$ A_t : = \{ x \in \R^n : (x, t) \in A\}. $$}
\end{remark}
\begin{lemma} [\cite{Valette}] \label{valette5} 
{\rm Let $\beta$ be a $j$-cycle and let $A \subset \R^n \times \R$ be a compact semi-algebraic family of sets with $\vert \beta \vert \subset A_t$ for any $t$. Assume that  $\vert \beta \vert$ bounds a $(j+1)$-chain in each $A_t$, $t >0$ small enough. Then $ \beta $ bounds a chain in $A_0$.}
\end{lemma}

\begin{definition}[\cite{Valette}] 

{\rm Given a subset $X \subset \R^{n}$, we define the {\it ``tangent cone at infinity''}, called ``{\it contour apparent \`a l'infini}''
in \cite{Thuy1} by: 
$$ C_{\infty}(X):=\{\lambda \in \S^{n-1}(0,1) \text{ such that } \exists \eta : (t_0, t_0 + \varepsilon] \rightarrow X \text { semi-algebraic,}$$
$$ \qquad \qquad \qquad \underset{t \rightarrow t_0}{\lim} \eta (t) = \infty, \underset{t \rightarrow t_0}{\lim}\frac{\eta(t)}{\vert \eta(t) \vert} = \lambda \}. $$
}
\end{definition}

\begin{lemma} [\cite{Thuy3}] \label{lemmathuy1}
{\rm Let $G=(G_1, \ldots , G_m) : \R^n \to \R^m$ be a polynomial mapping and $V$ the zero locus of   $\hat{G}: = (\hat{G_1}, \ldots, \hat{G}_m)$, where $\hat{G_i}$ is the leading form of  $G_i$. If $X$ is a subset of $\R^n$ such that $G(X)$ is bounded, then $C_\infty(X)$ is a subset of $\S^{n-1}(0,1) \cap V$, where $V = {\hat{G}}^{-1}(0)$.}
\end{lemma}


\begin{proof} [Proof of the Theorem \ref{cidinhathuy1}] 

%
%

Recall that in this case, $\dim_{\R} {\mathcal{V}}_G = 4$ (Proposition \ref{ProdimVG}) and ${\mathcal{V}}_G \setminus ({\mathcal{S}}_G \times \{ 0_{\R^{p}} \})$ is not smooth in general. Consider a stratification of ${\mathcal{V}}_G$, the strata of which are the strata of  ${\mathcal{S}}_G \times \{ 0_{\R^{p}} \}$ union the strata of the stratification of $K_0(F)$ defined by the  rank,  according to Thom \cite{Thom}. 
  Assume that $B(G) \neq \emptyset$, then by Remark \ref{remarkSG}, the set  ${\mathcal{S}}_G$ is not empty. This means that there exists a complex Puiseux arc  in ${\mathcal{M}}_G$ 
$$\gamma : D(0, \eta) \rightarrow \R^{6}, \quad \gamma = u z^{\alpha} + \ldots, $$
(with $\alpha$ negative integer and $u$
 is an unit vector of $\R^{6}$)  tending to infinity  such a  way that $G(\gamma)$ converges to a generic point $x_0 \in {\mathcal{S}}_G$. 
Then, the mapping $h_F \circ \gamma$, where $h_F = (F, \varphi_1, \ldots,\varphi_p)$ and $F$ is the restriction of $G$ on $ {\mathcal{M}}_G$ provides a singular $2$-simplex in ${\mathcal{V}}_G$ that we will denote by $c$. We prove now the simplex $c$ is $(\overline{t},2)$-allowable for total perversity $\overline{t}$.  In fact, by \cite{Ku}, is this case  we have $\dim_{\C} {\mathcal{S}}_G \leq 1$, then $\alpha = \codim_{\R} {\mathcal{S}}_G \geq 2$. The condition
$$ 0 = \dim_{\R} \{ x_0 \} = \dim_{\R} ( ({\mathcal{S}}_G \times \{ 0_{\R^{p}} \}) \cap \vert c \vert) \leq 2 - \alpha+ t_{\alpha},$$
implies $t_{\alpha} \geq \alpha - 2,$
with $\alpha \geq 2$, which is true for total perversity $\bar{t}$. Take now a stratum $V_i$ of  ${\mathcal{V}}_G \setminus ({\mathcal{S}}_G \times \{0_{\R^{p}}\})$. 
Denote by $\beta = \codim_{\R} V_i.$ If $\beta \geq 2$, we can choose the Puiseux arc $\gamma$ such that $c$ lies in the regular part of ${\mathcal{V}}_G \setminus ({\mathcal{S}}_G \times \{0_{\R^{p}}\})$. In fact, this comes from the generic position of transversality.  So $c$ is $(\overline{t},2)$-allowable in this case. We need to consider only the cases $\beta = 0$ and $\beta = 1$. We have the following two cases:

1) If $V_i$ intersects $c$, again by the generic position of transversality, we can choose the Puiseux arc $\gamma$ such that $0 \leq \dim_{\R} ( V_i \cap \vert c \vert) \leq 1$. The condition 
 $$\dim_{\R} ( V_i \cap \vert c \vert) \leq 2 - \beta+ t_{ \beta}$$
  holds since $2 - \beta+ t_{ \beta} \geq 1$, for $\beta = 0$ and $\beta = 1$. 

2) If $V_i$ does not meet $c$, then the condition 
$$ - \infty = \dim_{\R} \emptyset = \dim_{\R} ( V_i \cap \vert c \vert) \leq 2 - \beta+ t_{\beta}$$
 holds always. 

\noindent So the simplex $c$ is $(\overline{t},2)$-allowable for total perversity $\overline{t}$.

From here, the proof of the Theorem follows the ideas of \cite{Valette}:  We can always choose the Puiseux arc such that the support of $\partial c$ lies in the regular part  of ${\mathcal{V}}_G \setminus ({\mathcal{S}}_G \times \{0_{\R^{p}}\})$. 
 We have 
$$H_1(\Reg({\mathcal{V}}_G \setminus ({\mathcal{S}}_G \times \{0_{\R^{p}}\}))) = 0,$$ 
 then the chain $\partial c $ bounds 
a singular chain  
$e \in C^2(\Reg({\mathcal{V}}_G \setminus ({\mathcal{S}}_G \times \{0_{\R^{p}}\})) )$, where $e$ is a chain with compact supports or closed supports.
So $\sigma = c- e$ is a $(\overline{t},2)$-allowable cycle of  ${\mathcal{V}}_G$, with compact supports or closed supports.

We claim that $\sigma$ may not bound a  $3$-chain in ${\mathcal{V}}_G$.
Assume otherwise, {\it i.e.} assume that there is a  chain $\tau \in C_3({\mathcal{V}}_G)$, satisfying $\partial \tau=\sigma$. Let 
$$A:= h_F^{-1}(\vert \sigma \vert \cap ({\mathcal{V}}_G \setminus ({\mathcal{S}}_G \times \{0_{\R^{p}}\}))),$$ 
$$B:= h_F^{-1}(\vert \tau \vert \cap ({\mathcal{V}}_G \setminus ({\mathcal{S}}_G \times \{0_{\R^{p}}\}))).$$ 
By definition, $C_\infty(A)$ and $C_{\infty}(B)$ are subsets of $\S^{5}(0,1)$. 
 Observe that, in a neighborhood of infinity,  $A$ coincides with the support of the  Puiseux arc $\gamma$.  The set $C_\infty(A)$ is equal to $\S^1.a$ (denoting the orbit of $a \in \C^3$ under the action of $\S^1$ on $\C^3$, $(e^{i\eta},z) \mapsto  e^{i\eta}z$). Let $V$ be the zero locus of the leading forms  $\hat{G}: = (\hat{G_1}, \hat {G_2})$.  Since $G(A)$ and $G(B)$ are bounded, by  Lemma \ref{lemmathuy1}, $C_\infty(A)$ and $C_{\infty}(B)$ are subsets of $V \cap \S^{5}(0,1).$

\noindent For $R$ large enough, the sphere $\S^{5}(0, R)$ with center 0 and radius $R$ in $\R^{6}$ is transverse to $A$ and $B$ (at regular points). Let 
$$\sigma_R : = \S^{5}(0 , R) \cap A, \qquad \tau_R : = \S^{5}(0, R) \cap B.$$
 Then $\sigma_R$ is a chain bounding the chain $\tau_R$. 
Consider a semi-algebraic strong deformation retraction 
$\Phi : W \times [0;1] \rightarrow \S^1.a$, where $W$ is a neighborhood of  $\S^1.a$ in $\S^{5}(0,1)$ onto $\S^1.a$. Considering $R$ as a parameter, we have the following semi-algebraic families of  chains: 

1) ${\tilde{\sigma }}_R :=\frac{\sigma _R}{R},$
for $R$ large enough, then ${\tilde{\sigma}}_R$ is contained in $W$,

2)  $\sigma{'}_{R} = \Phi_1({\tilde{\sigma}}_R)$, where $\Phi_1(x) : = \Phi(x, 1), \qquad x \in W$,

3) $\theta_{R} = \Phi({\tilde{\sigma}}_R)$, we have 
 $\partial \theta_R  = \sigma'_R - {\tilde{\sigma}}_R ,$

4) $\theta{'}_{R} = \tau_R + \theta_R$, we have 
 $\partial \theta'_R  =  \sigma'_R.$

\noindent As, near infinity, $\sigma_R$ coincides with the intersection of the support of the arc $\gamma$ with $\S^{5}(0, R)$, for $R$ large enough the class of $\sigma'_R$ in $\S^1.a$ is nonzero.

Let $r = 1/R$, consider $r$ as a parameter, and let $\{ {\tilde{\sigma}}_r \}$, $\{ \sigma'_r \}$, $\{ \theta_r \}$ as well as $\{ \theta'_r \}$ the corresponding  semi-algebraic families of  chains. 

Denote by $E_r \subset \R^{6} \times \R$ the closure of $|\theta_r|$, and set $E_0 : = (\R^{6} \times \{ 0 \}) \cap E$. Since the strong deformation retraction  $\Phi$ is the identity on $C_{\infty}(A) \times [0,1]$, we see that  
$$E_0 \subset \Phi(C_{\infty}(A) \times [0,1]) = \S^1. a \subset V \cap \S^{5}(0,1).$$

Denote by $E'_r \subset \R^{6} \times \R$ the closure of $|\theta'_r|$, and set  $E'_0 : = (\R^{6} \times \{ 0 \}) \cap E'$. Since $A$ bounds $B$, then $C_{\infty}(A)$ is contained in $C_{\infty}(B)$. We have 
$$E'_0 \subset E_0 \cup C_{\infty}(B) \subset V \cap \S^{5}(0,1).$$

The class of $\sigma'_r$ in $\S^1.a$ is, up to a product with a nonzero constant, equal to the generator of $\S^1.a$. Therefore, since $\sigma'_r$ bounds the chain $\theta'_r$, the cycle $\S^1.a$ must bound a chain in $|\theta'_r|$ as well. By Lemma \ref{valette5}, this implies that $\S^1.a$ bounds a chain in $E'_0$ which is included in $V \cap \S^{5}(0,1)$.

 The set $V$ is a projective variety which is an union of cones in $\R^{6}$. Since   $\dim _{\C} V \leq 1$, so $\dim_{\R} V \leq 2 $ and  $\dim_{\R} V \cap \S^{5}(0,1) \leq 1.$
The cycle $\S^1 . a$ thus bounds a chain in $E'_0\subseteq V \cap \S^{5}(0,1)$, which is a finite union of circles, that provides a contradiction.
\end{proof}

%

Now we provides the proof of the Theorem \ref{cidinhathuy2}.

\begin{proof} [Proof of the Theorem \ref{cidinhathuy2}] 

The proof of this Theorem follows the idea of \cite{Thuy3} and the proof of Theorem \ref{cidinhathuy1}.

 Assume that $B(G) \neq \emptyset$. Similarly to the proof of the Theorem \ref{cidinhathuy1} and with the same notations in this proof but for the general case, we have:  since 
 $$\Rank_{\C} {(D \hat {G_i})}_{i=1,\ldots,n-1} > n-3$$
then 
$$ \corank_{\C}  {(D \hat {G_i})}_{i=1,\ldots,n-1} = \dim _{\C} V \leq 1,$$
 so $\dim_{\R} V \leq 2 $ and  $\dim_{\R} V \cap \S^{2n-1}(0,1) \leq 1.$
The cycle $\S^1 . a$ bounds a chain in $E'_0\subseteq V \cap \S^{2n-1}(0,1)$, which is a finite union of circles, that provides a contradiction. So we have 
$$IH_2^{\overline{t}, c}({\mathcal{V}}_G, \R) \ne 0, \quad IH_2^{\overline{t}, cl}({\mathcal{V}}_G, \R) \ne 0, \quad 
 H_2^{c}({\mathcal{V}}_G, \R) \ne 0  \text{ and } H_2^{cl}({\mathcal{V}}_G, \R) \ne 0.$$  
 From the  Goresky-MacPherson Poincar\'e  duality Theorem, we have
$$IH_2^{\overline{t}, c}({\mathcal{V}}_G, \R) = IH_{2n-4}^{\overline{0},cl}({\mathcal{V}}_G, \R) \text{ and }  IH_2^{\overline{t}, cl}({\mathcal{V}}_G, \R) = IH_{2n-4}^{\overline{0}, c}({\mathcal{V}}_G, \R),$$ 
that implies 
 $H_{2n-4}^{c}({\mathcal{V}}_G, \R) \ne 0$ and $H_{2n-4}^{cl}({\mathcal{V}}_G, \R) \ne 0$.
%
\end{proof}

\begin{remark}
{\rm 
The variety ${\mathcal{V}}_G$ associated to a polynomial mapping $G: \C^n \to \C^{n-1}$ is not unique, but the result of the theorems \ref{cidinhathuy1} and \ref{cidinhathuy2}  hold for any variety ${\mathcal{V}}_G$ among the constructed  varieties ${\mathcal{V}}_G$ associated to $G$. 

}
\end{remark}


With the conditions of Theorem \ref{cidinhathuy2}, the result of  \cite{VuiThang} also holds,  hence as a consequence of Theorem \ref{cidinhathuy2} in this paper 
and Theorems 2.1 and 2.6 in \cite{VuiThang}, we obtain the following corollary.

\begin{corollary} \label{coThuyCidinha3}
{\rm Let $G = (G_1, \ldots, G_{n-1}): \C^n \to \C^{n-1}$, where $n \geq 4$, be a polynomial mapping such that $K_0(G) = \emptyset$. Assume that the zero set 
$\{ z \in \C^n : \hat{G}_i (z) =0, i = 1, \ldots, n-1 \}$ has complex dimension one, where $\hat{G}_i$ is the leading form of $G_i$. 
 If the Euler characteristic of $G^{-1}(z^0)$ is bigger than that of the generic fiber, where $z^0 \in \C^{n-1}$, then 

\begin{enumerate}
\item[1)] $ H_2({\mathcal{V}}_G(\rho), \R) \ne 0,$ for any $\rho$,
\item[2)] $H_{2n-4}({\mathcal{V}}_G(\rho), \R) \ne 0$, for any $\rho$,
\item[3)] $IH_2^{\overline{t}}({\mathcal{V}}_G(\rho), \R) \ne 0,$ for any $\rho$, 
 where $\overline{t}$ is the total perversity.
\end{enumerate}
 }
\end{corollary}

\begin{proof} 
At first, since  the zero set 
$\{ z \in \C^n : \hat{G}_i (z) =0, i = 1, \ldots, n-1 \}$ has complex dimension one, then by the Theorem 2.6 in \cite{VuiThang}, any generic linear mapping $L$
 is  a very good projection with respect to any regular value $z^0$ of $G$. Then if the Euler characteristic of $G^{-1}(z^0)$ is bigger than that of the generic fiber, where $z^0 \in \C^{n-1}$, then by the Theorem 2.1 of \cite{VuiThang}, the set $B(G) \neq \emptyset$. 
Moreover, 
 the complex dimension of the set $\{ z \in \C^n : \hat{G}_i (z) =0, i = 1, \ldots, n-1 \}$ is the complex {\it corank} of ${(D \hat {G_i})}_{i=1,\ldots,n-1}$. Hence $\Rank_{\C} {(D \hat {G_i})}_{i=1,\ldots,n-1} = n-2$, and by the Theorem \ref{cidinhathuy2}, we finish the proof. 
\end{proof}

\begin{example}
{\rm Consider the suspension of the Broughton's example: 
$$G: \C^3 \to \C^2, \quad G(z, w, \zeta) = (z + z^2w, \zeta),$$
or, more general $G(z, w, \zeta) = (z + z^2w, g(\zeta))$ where $g(\zeta)$ is any polynomial of variable $\zeta$ and $g'(\zeta) \neq 0$. 
We can check that, for any function $\rho$, we have always $IH_2^{\overline{t}}({\mathcal{V}}_G, \R) \neq 0$.
}
\end{example}

\begin{remark}
{\rm  The condition $B(G) = \emptyset$ does not imply $IH_2^{\overline{t}}({\mathcal{V}}_G, \R) = 0$, since in this case ${\mathcal{S}}_G$ maybe not empty.
}
\end{remark}

\begin{example}
{\rm Let 
$$G: \C^3 \to \C^2, \quad G(z, w, \zeta) = (z, z\zeta^2 + w).$$

 1) If we choose the function $\rho = \vert \zeta \vert ^2,$
then  ${\mathcal{S}}_G = \emptyset$ and $IH_2^{\overline{t}}({\mathcal{V}}_G, \R) = 0$.

2) If we choose the function 
$\rho = \vert w \vert ^2,$
then  ${\mathcal{S}}_G \neq \emptyset$ and $IH_2^{\overline{t}}({\mathcal{V}}_G, \R) \neq 0$.
}
\end{example}

%
%
%
%
%

\bibliographystyle{plain}

\end{document}